\documentclass[12pt, reqno]{amsart}
\usepackage{amssymb,latexsym,amsmath,amscd,amsfonts}
\usepackage{latexsym}
\usepackage[mathscr]{eucal}
\usepackage{bm}
\usepackage{hyperref}
\usepackage{cleveref}
\usepackage{enumerate}
\usepackage{xcolor}

9.5in \numberwithin{equation}{section}

\newcommand{\beg}{\begin{equation}}
    \newcommand{\eeg}{\end{equation}}
\newcommand{\ben}{\begin{eqnarray*}}
    \newcommand{\een}{\end{eqnarray*}}

\newtheorem{thm}{Theorem}[section]
\newtheorem{cor}[thm]{Corollary}
\newtheorem{lem}[thm]{Lemma}

\numberwithin{equation}{section} \theoremstyle{definition}
\newtheorem{defn}[thm]{Definition}
\newtheorem{rem}[thm]{Remark}

\newtheorem{eg}[thm]{Example}

\makeatletter \@namedef{subjclassname@2020}{\textup{2020}
Mathematics Subject Classification} \makeatother

\begin{document}

\title[Peripheral Poisson Boundary]{ Peripheral Poisson Boundary: Extensions and Examples}

  \author[Bhat and Dias]{B V Rajarama Bhat and Astrid Swizell Dias}
    \address[Statistics and Mathematics Unit, Indian Statistical Institute, R V College Post, Bangalore - 560059, India]{}
\email{bvrajaramabhat@gmail.com}
    \address[Institute of Mathematical Sciences, Taramani, Chennai]{}
\email{astriddias97@gmail.com}

\begin{abstract}

The main purpose of this article is to explore the possibility of
extending the notion of peripheral Poisson boundary of unital completely positive (UCP)  maps to
contractive completely positive (CCP) maps and to unital and non-unital contractive
quantum dynamical semigroups on von Neumann algebras. We observe that the
theory extends easily in the setting of von Neumann algebras and
normal maps. Surprisingly, the peripheral Poisson boundary is unital, whenever it is nontrivial, even for contractive semigroups.  The strong operator limit formula for computing the extended Choi-Effros product remains intact. However, there are serious
obstacles in the framework of $C^*$-algebras, and we are unable to define the extended Choi-Effros product in such generality. We provide several intriguing examples to
illustrate this.

    \end{abstract}

    \maketitle

\section{Introduction}

Consider a normal unital completely positive (UCP) self-map $\tau $ on a von Neumann algebra $\mathcal{A}$. The space $F(\tau )$ of fixed points of $\tau $,
becomes a von Neumann algebra under a modified product called the Choi-Effros product. Taking inspiration from classical probability, M. Izumi (\cite{Izumi2002},\cite{Izumi2004})  called this algebra (or its explicit representation) as the {\em Poisson boundary} of $\tau$. This has found several  applications in operator algebra theory \cite{DP}. In quantum theory of open systems it is important to study not only individual completely positive maps but also the associated dynamics (see \cite{AL, Da, Fa, GS}). This means that we are looking at either discrete or one parameter semigroups of unital or contractive completely positive maps. To understand the dynamics of $\tau$, that is, to study the semigroup $\{\tau ^n: n\in \mathbb{Z}_+\}$, it becomes crucial to analyze the space $E(\tau )$ consisting of the linear span of peripheral eigenvectors (those with eigenvalues having modulus 1). This has been observed many times before (see \cite{RA, FOR, G}). Note that the operator space $E(\tau )$ has fixed points space $F(\tau )$ as a subspace. Bhat, Kar and Talwar \cite{BKT1} noted that the Choi-Effros product can be extended to this space and the norm closure becomes a unital $C^*$-algebra. This was called as {\em peripheral Poisson boundary} of $\tau $ and the basic properties were derived in \cite{BKT1}. Some further studies have been carried out in \cite{BBDR} and \cite{BKT2}.

The theory of peripheral Poisson boundary was introduced in \cite{BKT1} for unital 
completely positive maps (UCP) on von Neumann algebras. The main tool was the {\em minimal dilation
theory\/}  developed in \cite{Bha96}, \cite{Bha99}, \cite{Bh3}, \cite{BP95} and \cite{BS}.  One may try to extend the concept of peripheral Poisson boundary to more general settings.  Here we explore the theory for contractive completely positive maps and for one-parameter semigroups of unital/contractive completely positive maps. These settings are perfectly natural because of their importance in the quantum theory of open systems.  The extension has been carried out in Sections 3 and 4, in von Neumann algebra settings (see Theorem \ref{Main 3}, \ref{Main 4} and subsequent definitions). We see that the concept of peripheral Poisson boundary and most of the related structure theorems extend naturally without any difficulties. One of the crucial upshots is that in the contractive case, the peripheral Poisson boundary remains unital, if it is non-trivial. This is somewhat unexpected and has important connections with Perron-Frobenius theory (Theorem \ref{PP}).

The framework of peripheral Poisson boundary as developed in \cite{BKT1} has a peculiarity. It  starts with a normal CP map on a von Neumann algebra, but the peripheral Poisson boundary constructed is only a $C^*$-algebra. In fact, through an example using $\lambda $-Toeplitz operators, it was shown in \cite{BKT1} that, in general, we cannot get a von Neumann algebra. In other words, the Choi-Effros product may not extend to the strong closure of the span of peripheral eigenvectors and we have to remain satisfied with a $C^*$-algebra. However, the question remained as to whether the theory can be developed completely in the $C^*$-algebra setting, i.e, starting with a unital CP map on a $C^*$-algebra. As mentioned before, \cite{BKT1} makes extensive use of dilation theory. Therefore, the first question in reader's mind is perhaps about the existence of minimal dilation of unital CP maps to $*$-endomorphisms in $C^*$-algebra setting. Actually, this is feasible and  in fact it was discovered before the theory got developed for von Neumann algebras \cite{Bha99}. In a way, the minimal dilation constructions are smoother in $C^*$-setting as the norm topology is easier to work with.   In the von Neumann algebra setting, one had to handle the topology with more care (see Bhat and Skeide \cite{BS}). 

The conclusion of the previous paragraph is that dilation theory may not be an obstruction to develop a theory of peripheral Poisson boundary for unital CP maps on $C^*$-algebras. However, all the constructions of the Choi-Effros type product in \cite{BKT1} used the strong operator topology and didn't allow us to extend the arguments to $C^*$-algebra setting. To our astonishment, finally we have some counter examples (see Section 5), which demonstrate conclusively that in general it is impossible to make the norm closed operator system generated by peripheral eigenvectors into a $C^*$-algebra. So we need to start with the von Neumann algebra setting of normal maps even though we finally end up with the construction of only a $C^*$ algebra and not a von Neumann algebra.  The counter examples of Section 5 could be termed as elementary, but as is usual in such cases, it took considerable effort to arrive at them for the first time.

Section 2 has some background information and preliminaries. In
particular, we recall the minimal dilation theorem and the construction
of the peripheral Poisson boundary for unital completely positive maps
on von Neumann algebras.

In Section 3, we use the unitization procedure to extend the theory
of the peripheral Poisson boundary to contractive completely positive
maps on von Neumann algebras. In the next Section, the theory has been extended to one parameter
semigroups of unital completely positive maps on von Neumann
algebras and then subsequently to contractive one parameter
semigroups. Finally, in Section 5, we have examples showing the failure of peripheral Poisson boundary theory for maps on $C^*$-algebras.

\section{Preliminaries}

To begin with, we recall the definition and construction of peripheral Poisson boundary as described in \cite{BKT1}. This also helps us to set up some notations.

For $d\in \mathbb{N},$ $\mathbb{M}_d$ will denote the algebra of $d\times d$ complex matrices, considered as the space of operators on the Hilbert space $\mathbb{C}^d$ in the natural way.
Our Hilbert spaces are all complex and separable,
whose inner products, denoted by $\langle \cdot , \cdot \rangle $ are anti-linear
in the first variable. By $\mathscr{B}(\mathcal{H})$ we denote the
algebra of all bounded operators on a Hilbert space $\mathcal{H}$. The identity of a Hilbert space $\mathcal{H}$ is denoted by $1_{\mathcal{H}}$. Similarly, we may denote the identity of an algebra $\mathcal{A}$ by $1_{\mathcal{A}}.$ If the Hilbert space or the algebra under consideration is clear from the context, we may drop the subscript.

 Suppose $\mathcal{A}\subseteq \mathscr{B}(\mathcal{H})$ is a von
Neumann algebra,  and $\tau: \mathcal{A}\to \mathcal{A}$ is a normal,
unital completely positive map.   Consider the set of fixed points of $\tau $:
$$F(\tau):=\{x\in \mathcal{A}: \tau(x)=x \}.$$ 
Clearly $F(\tau )$ is an operator system, that is, it is a subspace closed under adjoints and contains the identity.  In general, it is not closed under taking products. However, it admits a new product called the Choi-Effros product and this makes $F(\tau )$ a von Neumann algebra. This is known as the {\em Poisson boundary of $\tau.$} Note that $F(\tau )$ is the eigenspace of $\tau$ with respect to eigenvalue 1. Taking a cue from this, for $\lambda \in \mathbb{T}:=\{z: z\in\mathbb{C}, |z|=1\}$, take
$$E_\lambda (\tau)=\{x \in \mathcal{A}: \tau(X)=\lambda x \}.$$  So $E_1(\tau )=F(\tau ).$ The elements of $E_{\lambda }$ with $\lambda \in\mathbb{T}$ are known as {\em peripheral eigenvectors\/} of $\tau .$ 
Define
$$E(\tau)=~\mbox{span}\{x\in \mathcal{A}: x\in E_\lambda(\tau) ~~\mbox{ for  some } \lambda \in \mathbb{T}\}.$$

In studying {\em `quantum dynamics'},\/  in discrete time, one is interested in analyzing the behavior of the discrete semigroup $\{ \tau ^n: n\in \mathbb{Z}_+\}.$ In general $F(\tau ^n)$ varies with $n$. However, an elementary application of the spectral mapping theorem shows (see \cite{BKT1}):
\begin{equation}
E(\tau ^n)= E(\tau ), ~~\forall n\geq 1.
\end{equation}
Note that $\tau $ being continuous, $F(\tau )$ is closed. However, this need not be the case with $E(\tau ).$ Take 
$$\mathcal{P}(\tau ) =\overline{E(\tau )}^{\|\cdot \|},$$
the norm closure of $E(\tau ).$ 

Given a triple $(\mathcal{H}, \mathcal{A}, \tau )$ as above, there exists a triple $(\mathcal{K}, \mathcal{B}, \theta )$ called  the {\em minimal dilation} of $\tau$, where (i) $\mathcal{K}$ is a Hilbert space containing $\mathcal{H}$, (ii) $\mathcal{B}\subseteq \mathscr{B}(\mathcal{K})$ is a von Neumann algebra  satisfying $\mathcal{A}= p\mathcal{B}p$, where $p$ is the projection of $\mathcal{K}$ onto $\mathcal{H}$,  and (iii) $\theta : \mathcal{B}\to \mathcal{B}$, is a unital $*$-endomorphism `dilating' $\tau$:
\begin{equation}\label{dilation}\tau ^n(x)=p\theta ^n(x)p, ~~n\geq 0, x\in \mathcal{A}.\end{equation}
Here and elsewhere in this dilation theory, the operators $x\in \mathscr{B}(\mathcal{H})$ are identified with $pxp$ in $\mathscr{B}(\mathcal{K})$ for $\mathcal{H}\subseteq \mathcal{K}$. The exact defining properties of this dilation that make it unique up to unitary equivalence will be described in the next section in a more general setting. One of the consequences of these properties in the unital case is that $\{\theta^n(p):n\in \mathbb{Z}_+\} $ is an increasing family of projections that increases to identity as $n$ tends to infinity. 

As $\theta $ is also a completely positive map in its own right, we already have $E_{\lambda }(\theta ), E(\theta ), \mathcal{P}(\theta )$ well defined by simply replacing $\tau $ by $\theta.$ One of the major observations of \cite{BKT1} is that every peripheral eigenvector of $\tau$ `lifts' uniquely to a peripheral eigenvector of $\theta$, and this lift can be extended to the
whole of $\mathcal{P}(\tau )$.  More specifically,  given $x\in \mathcal{P}(\tau )$, there exists a unique $\widehat{x}$ in $\mathcal{P}(\theta )$ such that 
\begin{equation}\label{lift}x=p\widehat{x}p.\end{equation}
There is an explicit formula for the lift $\widehat{x}$. If $x\in E_\lambda(\tau )$ (with $\lambda \in \mathbb{T}$),
\begin{equation}\label{lift formula}\widehat{x} = ~\mbox{s.}\lim _{n\to \infty}\frac{\theta ^n(x)}{\lambda ^n},\end{equation}
where `s.' indicates that the limit is taken in strong operator topology. The mapping $x\mapsto \widehat{x}$ from $\mathcal{P}(\tau )$ to $\mathcal{P}(\theta )$ is a bijection. In fact,   it is a complete isometric isomorphism of operator systems. The inverse map of this is nothing but the compression map $z\mapsto pzp$, restricted to $\mathcal{P}(\theta ).$  Now $\theta $ being a unital $*$-endomorphism, $\mathcal{P}(\theta )$ is seen to be a unital $C^*$ subalgebra of $\mathcal{B}$.  This allows us to define a product, denoted `$\circ $', on $\mathcal{P}(\tau )$, by taking
\begin{equation}\label{product}x\circ y = p(\widehat{x}.\widehat{y})p.\end{equation}
This is called the extended Choi-Effros product. With this product $\mathcal{P}(\tau )$ is a unital $C^*$-algebra and  $(\mathcal{P}(\tau ), \circ )$ is defined as the {\em peripheral Poisson boundary\/} of $\tau $. It is important to note that $\tau $ leaves $\mathcal{P}(\tau )$ invariant and $\tau $ restricted to $\mathcal{P}(\tau )$ is an automorphism of $(\mathcal{P}(\tau), \circ ).$  For $x\in E_{\lambda }(\tau ),  y\in E_{\mu } (\tau )$ where $\lambda, \mu \in \mathbb{T}$, 
we have the formula: \begin{equation}\label{limit formula} x\circ y = ~\mbox{s.}\lim _{n\to \infty}\frac{\tau ^n (xy)}{\lambda ^n\mu ^n }.\end{equation}
In principle, we could have used this to define the extended Choi-Effros product. However, as of now, we do not know how to prove the existence of these strong limits without using dilation theory. 

\section{CCP maps on von Neumann algebras}

The primary goal of this Section is to extend the concept of peripheral
Poisson boundary to contractive (non-unital) completely
positive maps.  We begin with recalling the minimal dilation theorem for contractive completely positive maps on von Neumann algebras. This includes the unital case as a special case.  

\begin{thm}\label{Minimal} {\em (\cite{Bha96}, \cite{Bha99}, \cite{BS} )}
Let $(\mathcal{H}, \mathcal{A}, \tau )$ be a triple where $\mathcal{H}$ is a Hilbert space, $\mathcal{A}
\subseteq \mathscr{B}(\mathcal{H})$, is a von Neumann algebra and $\tau:\mathcal{A}\to \mathcal{A}$ is a normal contractive
completely positive map. Then there exists a triple
$(\mathcal{K}, \mathcal{B}, \theta )$ with following properties:
\begin{itemize}
\item[(i)]  $\mathcal{K}$ is a Hilbert space such that $\mathcal{H}\subseteq \mathcal{K}.$
\item[(ii)] $\mathcal{B}\subseteq \mathscr{B}(\mathcal{K})$ is a von
Neumann algebra such that $\mathcal{A}=p\mathcal{B}p$ where $p$ is
the orthogonal projection of $\mathcal{K}$ to $\mathcal{H}.$
\item[(iii)] $\theta :\mathcal{B}\to \mathcal{B}$ is a normal 
$*$-endomorphism.
\item[(iv)] $\tau ^n(x)= p\theta ^n(x)p, ~~\forall x\in \mathcal{A},
n\in \mathbb {Z}_+.$
\item[(v)] $ \mathcal{K}= \overline{span}\{ \theta^n(x_n)\theta^{n-1}(x_{n-1})...\theta(x_1)h: h\in \mathcal{H}, \  x_1, ...,x_n \in \mathcal{A}, \ 
 n\geq 0 \}. $
\item[(vi)] $\mathcal{B}$ is the von Neumann algebra generated by $\{\theta^n(x): n\in \mathbb{Z}_+, x\in \mathcal{A} \}$.
\end{itemize}
\end{thm}

We call $(\mathcal{K}, \mathcal{B}, \theta )$ as the minimal dilation of the triple $(\mathcal{H}, \mathcal{A}, \tau )$. The action of $\theta ^n$ on $x$ in (iv) is meaningful as following our convention $x$ is identified with $pxp$. The minimal dilation is unique up to unitary equivalence. This theorem is well-known in the unital case, and it was first proved when the algebra is $B(\mathcal{H} )$ in \cite{Bha96} and in the $C^*$- algebra case in \cite{Bha99}. For the von Neumann algebra situation, we refer to \cite{BS}. Most of the focus in these papers are on the continuous case consisting of one parameter semigroups of completely positive maps, and this dilation theorem has found extensive applications in the construction of $E_0$-semigroups  using CP semigroups \cite{Powers}. The discrete case of powers of a single map follows much more easily and is often treated as a first step for the continuous situation. Alternative presentations in the von Neumann algebra setting can be seen in \cite{DP} and \cite{Sawada}.
A different perspective using von Neumann correspondences is present in \cite{MS}. 

The extension of the minimal dilation theorem to the contractive case is easy through a unitization trick, and this has been detailed in \cite{BP95} and \cite{Bha99}. For the reader's convenience, we outline this procedure here.

Let $\mathcal{A}$ be a von Neumann algebra and let
$\tau:\mathcal{A}\to \mathcal{A}$ be a normal contractive completely
positive map.  Note that contractivity of $\tau $ implies that
$\tau(1) \leq 1$. As the unital case has already been dealt-with before, for the sake of clarity, in the following we assume $\tau (1)\neq 1.$  Consider the von Neumann algebra
$\tilde{\mathcal{A}}:=\mathcal{A}\oplus\mathbb{C}$ acting on the Hilbert space $\tilde{\mathcal{H}}:=\mathcal{H}\oplus \mathbb{C}$ in the natural way. Define
$\tilde{\tau }:\tilde{\mathcal{A}} \to \tilde{ \mathcal{A}}$
as follows:
$$ \tilde{\tau }(x\oplus c) = [\tau(x)+ c(1-\tau (1))]\oplus c, ~~ x\oplus c\in \mathcal{A}\oplus \mathbb{C}.
$$
We consider $\tilde{\tau } $ as `unitization'  of $\tau .$ Note that $(\tilde {\tau })^n(x\oplus 0)=\tau^n(x)\oplus 0$. It is easy to see
that $\tilde{\tau }$ is a unital completely positive extension of $\tau$. Now
the minimal dilation of $\tau $ can be got by considering the
minimal dilation of $\tilde{\tau } $  and restricting the attention to a portion of it  as follows (see \cite{Bha99}). 

If $(\tilde{\mathcal{K}}, \widetilde{\mathcal{B}}, \tilde{\theta })$ is the minimal dilation of $(\tilde{\mathcal{H}}, \tilde{\mathcal{A}}, \tilde{\tau })$, $\tilde{\mathcal{H}}=\mathcal{H}\oplus \mathbb{C}$ is a subspace of the Hilbert space $\tilde{\mathcal{K}}$. Take $\mathcal{K}$ as $(0\oplus \mathbb{C})^\perp $ in $\tilde{\mathcal{K}}.$  
Then we have the natural decomposition $\tilde{\mathcal{K}}=\mathcal{K}\oplus \mathbb{C}$. It turns out that in a natural way,  $\widetilde{\mathcal{B}}
=\mathcal{B}\oplus \mathbb{C}$, and $$\tilde{\theta }(b\oplus c)=[\theta (b)+c(1-\theta (1))]\oplus c,~~b\oplus c\in \mathcal{B}\oplus \mathbb{C},$$ on defining $\mathcal{B}$ and $ \theta$ appropriately.  Then
$(\mathcal{K}, \mathcal{B}, \theta )$ is the minimal dilation of $(\mathcal{H}, \mathcal{A}, \tau )$, satisfying all the conditions of Theorem \ref{Minimal}.  In other words, the minimal dilation theorem holds for contractive completely positive maps where the dilated maps are  general (not-necessarily unital) normal $*$-endomorphisms. In the unital case, $\theta ^n(p)$ increases to the identity map on the dilation space $\mathcal{K}$ as $n$ increases to infinity. In the non-unital case, $\theta ^n(p)$ is not increasing and $\theta$ is non-unital.

In the following, let $(\mathcal{K}, \mathcal{B},\theta)$ be the minimal dilation
of $(\mathcal{H}, \mathcal{A}, \tau)$ in the von Neumann algebra setting as above for a contractive non-unital completely positive map and let $p$ be the projection of $\mathcal{K}$ onto $\mathcal{H}$.  Let us see the relationship between the peripheral eigenvectors  of $\tau $
and its unital extension $\tilde{\tau } .$  Note that for $\lambda \neq 1$, if
$\tilde{\tau} (x\oplus c)=\lambda(x\oplus c)$. Then $\lambda (x\oplus c)=[\tau(x)+
c(1-\tau(1))]\oplus c$, and this implies $c=0$. Therefore, for $\lambda \neq
1$, the eigenvectors of $\tilde{\tau }$ essentially coincide with that of
$\tau$.

Notice that if $y$ is
a fixed point of $\tau$, then $y\oplus 0$ is a fixed point of $\tilde{\tau }$.  Also $1_{\mathcal{H}}\oplus 1$ is anyway a fixed point of $\tilde{\tau }$. It is easy to see that all fixed points of $\tilde{\tau }$ are in the linear span of such vectors. 
Consequently, \begin{equation}\label{eigenvectors}
E(\tilde{\tau}) =\mbox{span} (\{y\oplus 0: y\in E(\tau )\}\bigcup \{1_{\mathcal{H}}\oplus 1\}).
\end{equation}
In other words, the peripheral eigenvectors of $\tilde{\tau}$ are those of $\tau$ supplemented by $1_{\mathcal{H}}\oplus 1.$

Now we extend the lifting property of peripheral eigenvectors to the non-unital case. We recall that for all $n \in \mathbb{Z_+}$,
$1_{\mathcal{H}}\oplus1 \leq (\tilde{\theta}) ^n(1_{\mathcal{H}}\oplus 1)$ and $(\tilde{\theta}) ^n(1_{\mathcal{H}}\oplus 1) \uparrow
1_{\tilde{\mathcal{K}}}$ in strong operator topology. 
Note that $\tau(1_{\mathcal{H}})\neq 1_{\mathcal{H}}$, ensures that $\theta (1_{\mathcal{K}})\neq 1_{\mathcal{K}} $ and therefore $1_{\mathcal{K}}$ is not a fixed point of $\theta $.

\begin{lem} Suppose  $0\neq x \in \mathcal{A}$ is a peripheral eigenvector of $\tau $, that is,  $\tau(x)=\lambda x$ for some  $\lambda \in \mathbb{T} $. 
Then there exists a unique $\widehat{x}\in \mathcal{B}$, such that $\theta(\widehat{x})=\lambda \widehat{x}$, satisfying $p\widehat{x}p=x$. Moreover,   $\widehat{x}$ is given by \begin{equation}\label{lifting} \widehat{x} = s. \lim_{n \to \infty} \frac{\theta^n(x)}{\lambda^n}.\end{equation}
\end{lem}
\begin{proof} Without loss of generality we may assume that the dilation is got through the unitization procedure described above.
Then  $x\oplus 0 $ is a peripheral eigenvector of $\tilde{\tau }$ and by (\ref{lift formula}) it has unique lifting
$\widehat{x\oplus 0}$ in $\tilde{K}$, satisfying $\tilde{\theta}(\widehat{x\oplus 0})=\lambda (\widehat{x\oplus 0})$ and is given by
$$\widehat{x\oplus 0}= s.\lim _{n\to \infty}\frac{(\tilde{\theta})^n(x\oplus 0)}{\lambda ^n}.$$
Since $(\tilde{\theta})^n(x\oplus 0) =\theta ^n(x)\oplus 0$, it follows that $\widehat{x\oplus 0}$ is of the form $\widehat{x}\oplus 0$, for some $\widehat{x}$ 
and  also (\ref{lifting}) holds.

 If $y\in \mathcal{B}$ and $\theta (y) =\lambda y$ then $\tilde{\theta }(y\oplus 0)= \lambda (y\oplus 0)$, and $y\oplus 0$ compresses to an element of the form $x\oplus 0$.
 Hence the uniqueness also follows from the uniqueness of lifting in the unital case.
\end{proof}

As in the unital case the compression map $z\mapsto pzp$ becomes the inverse of the lifting map, and we have the following Theorem.
   
\begin{thm}\label{Main 3}
Let $\mathcal{A} \subseteq \mathscr{B}(\mathcal{H})$ be a von Neumann
algebra and $\tau:\mathcal{A}\to \mathcal{A}$ be a normal
contractive completely positive map. Let
$(\mathcal{K},\mathcal{B},\theta)$ be the minimal dilation of $\tau$ as above and let $p$ be the projection of $\mathcal{K}$ to $\mathcal{H}.$ Then the compression map $z\mapsto pzp$ from $\mathcal{B}$ to $\mathcal{A}$ restricted to $\mathcal{P}(\theta )$ is a complete isomorphism between $\mathcal{P}(\theta )$ and $\mathcal{P}(\tau ). $
\end{thm}
\begin{proof}
Making use of the unitization of $\tau $ and the lift of peripheral eigenvectors described above, the proof is exactly as in the unital case \cite{BKT1}, although $\mathcal{P}(\theta )$  does not contain the unit of $\mathcal{B}$ unless $\tau$ is unital.
\end{proof}

Now, as before, we set $x\circ y= p\widehat{x}.\widehat{y}p$ for $x,y$ in $\mathcal{P}(\tau )$ and this makes
$\mathcal{P}(\tau )$ a $C^*$-algebra.

\begin{defn} Let $\mathcal{A} \subseteq \mathscr{B}(\mathcal{H})$ be a
von Neumann algebra, and $\tau:\mathcal{A} \to \mathcal{A}$ be a 
normal contractive completely positive map. Then the $C^*$- algebra
$\mathcal{P}(\tau)$ constructed above is called the {\em peripheral Poisson boundary}  of
$\tau$ and the product $`\circ '$ is called the {\em extended Choi-Effros product}. \end{defn}

The following result follows exactly as in the unital case, and so we
omit the proof.

\begin{rem}
For $x,y \in \mathcal{A}$, with $\tau(x)=\lambda x , \tau(y)=\mu y,$ where $\lambda, \mu \in \mathbb{T}$,

      $$ x\circ y= s.\lim_{n \to \infty} \frac{\tau^n(xy)}{\lambda^n\mu^n}.$$
\end{rem}

\begin{rem}
  If the von Neumann algebra $\mathcal{A}$ is abelian, then the peripheral Poisson boundary of $\tau:\mathcal{A} \to \mathcal{A}$ is also abelian.
\end{rem}
\begin{proof} Clear from the previous Remark.
\end{proof}

\begin{rem}
  Let $\tau:\mathcal{A}\to \mathcal{A}$ be a normal, contractive, completely positive map. Then $x \to \tau(x)$ is an automorphism of the peripheral Poisson boundary $(\mathcal{P}(\tau), \circ )$.
\end{rem}
\begin{proof}
This follows from the Theorem \ref{Main 3} and the fact that $\theta $ is an automorphism of $E(\theta ).$ (See the proof of Theorem 2.12 in \cite{BKT1} and observe that it does not require unitality of $\theta $).
\end{proof}

When $\tau $ is non-unital, it is possible that it has no peripheral eigenvectors. In such a case $\mathcal{P}(\tau )$ is just the trivial $C^*$-algebra $\{0\}.$ If $\tau $ does admit some peripheral eigenvectors then $(\mathcal{P}(\tau ), \circ )$ is a non-trivial $C^*$-algebra. It is interesting that it is always unital. This is not obvious. It also raises the question as to what is the unit. Obviously it has to be a special fixed point. This connects to the theory of Perron-Frobenius eigenvectors.

\begin{defn}
Let $\tau $ be a positive map on a von Neumann algebra $\mathcal{A}$. Then $x\in \mathcal{A}$ is said to be a {\em Perron-Frobenius eigenvector\/} of $\tau ,$ if $x\geq 0, \|x\|=1$, and $\tau (x)= rx$, where $r$ is the spectral radius of $\tau$. It is said to be a {\em maximal Perron-Frobenius eigenvector\/}  if $y\leq x$ for every Perron-Frobenius eigenvector $y$ of $\tau$.
\end{defn}
Under some restrictive conditions on the CP map such as irreducibility, the Perron-Frobenius eigenvector is unique. This is not the case, in general. However, it is obvious that whenever a maximal Perron-Frobenius eigenvector exists, it is unique.

\begin{lem}
    Let $\tau$ be a normal contractive completely positive map on a von Neumann algebra $\mathcal{A}\subseteq \mathscr{B}(\mathcal{H})$. Let $(\mathcal{K},\mathcal{B},\theta)$ be the minimal dilation for $\tau$ as above. Then $\theta(1_{\mathcal{K}}-p)\leq (1_{\mathcal{K}}-p)$.
\end{lem}
\begin{proof}
    Let $(\Tilde{\mathcal{K}},\Tilde{\mathcal{B}}, \Tilde{\theta})$ be the minimal dilation of unitization of $\tau$ : $(\Tilde{\mathcal{H}},\Tilde{\mathcal{A}},\Tilde{\tau})$, where $\Tilde{\mathcal{H}}=\mathcal{H}\oplus \mathbb{C}, \ \mbox{and} \ \Tilde{\mathcal{A}}=\mathcal{A}\oplus \mathbb{C}$ and $\Tilde{\tau}(x\oplus c)=(\tau(x)+c(1_{\mathcal{A}}-\tau(1_{\mathcal{A}})))\oplus c$. Note that $1_{\Tilde{\mathcal{K}}}=1_{\mathcal{B}}\oplus 1_{\mathbb{C}}=1_{\mathcal{K}}\oplus 1_{\mathbb{C}}=s. \lim_{n \to \infty}(\Tilde{\theta})^n (p\oplus 1_{\mathbb{C}})$, where $p $ is the projection of $\mathcal{K}$ onto $\mathcal{H}.$
Also $p\oplus 1_{\mathbb{C}} \leq \Tilde{\theta}(p\oplus 1_{\mathbb{C}})=(\theta(p)+(1_{\mathcal{K}}-\theta(1_{\mathcal{K}})))\oplus 1_{\mathbb{C}}$. Hence $\theta(1_{\mathcal{K}}-p)\leq (1_{\mathcal{K}}-p)$.
\end{proof}

\begin{thm}\label{PP} Let $\tau $ be a contractive normal completely positive map on a von Neumann algebra $\mathcal{A}$. Take $q_{\tau }:=\mbox{s-}\lim _{n\to \infty}\tau ^n(1_{\mathcal{A}}).$ Then  the following are equivalent: (i) $q_{\tau }\neq 0$; (ii) $\tau$ has a non-trivial fixed point; (iii) The peripheral Poisson boundary $\mathcal{P}(\tau )$ is non-trivial;  (iv) $\mathcal{P}(\tau )$ is a unital $C^*$-algebra. 
In such a case,  $q_{\tau }$ is the unit of $\mathcal{P}(\tau )$, and it is also the unique maximal Perron-Frobenius eigenvector of $\tau .$
\end{thm}
\begin{proof}
If $q_\tau =0$, then $\tau ^n(x)$ converges to $0$ for every $x$ and hence $\tau $ can't admit any peripheral eigenvectors. So, let us take $q_\tau \neq 0$. Then $\tau(q_\tau )= ~\mbox{s.}\lim_{n\to \infty}\tau^{n+1}(1_{\mathcal{A}})=q_\tau.$ Therefore $q_\tau $ is a fixed point of $\tau $ and hence $\mathcal{P}(\tau )$ is non-trivial. This shows the equivalence of (i), (ii), and (iii).

Assume (iii). To see that $\mathcal{P}(\tau )$ is a unital $C^*$-algebra, recall that $\mathcal{P}(\tau)$ is in an isometric bijection with $\mathcal{P}(\theta)$. 
Since $\theta^n(I)$ is a projection, where $ I=1_{\mathcal{K}}$ is the identity of the dilation space, clearly $q_{\theta}=~\mbox{s.}\lim_{n\to\infty}\theta^n(I)$ is also a projection and is a fixed point of $\theta$. From the $*$-endomorphism property of $\theta $, it is clear that $\mathcal{P}(\theta )$ is a unital $C^*$-algebra with $q_\theta$ as its identity.

Now, we show that $p q_\theta p=q_\tau$, which would imply that $q_\tau$ is the identity in $\mathcal{P}(\tau)$. Note that $p$ is identified with the identity of $\mathcal{A}.$
Then, $$p\theta^n(z)p\oplus 0=(p\oplus 1_{\mathbb{C}})(\tilde{\theta})^n (z \oplus 0) (p\oplus 1_{\mathbb{C}}) = (\tilde{\tau})^n(pzp\oplus 0), \ \forall z \in \mathcal{B}.$$ This implies $p\theta^n(z)p=\tau^n(pzp)$. Taking $z=I$ and the strong limit on both sides, we have $q_\tau=~\mbox{s.}\lim_{n\to\infty}\tau^n(1_{\mathcal{A}})$ $   =~\mbox{s.}\lim_{n\to\infty}p\theta^n(I)p=pq_\theta p.$ So (iii) implies (iv).
Finally, (iv) implies (iii), is trivial.
\end{proof}

\begin{rem} In proving  $pq_\theta p= q_\tau$, we have made non-trivial  use  of the following property of the minimal dilation of quantum Markov semigroups: $$p\theta ^n(z)p=\tau ^n(pzp), ~~\forall n\geq 0, z\in \mathcal{B}.$$
This feature plays a crucial role in the general theory of dilation. Dilations satisfying such a condition were called as {\em strong dilations\/} by Shalit and Skeide \cite{SS}. For minimal dilations of quantum Markov semigroups it is a consequence of the fact: $\theta (1_{\mathcal{K}}-p) \leq (1_{\mathcal{K}}-p).$ Earlier, its importance was noted in \cite{Bh3} and was called as {\em regularity}.
\end{rem}

The unilateral right-shift on $l^{\infty}(\mathbb{Z}_+)$ shows that a contractive completely positive map with unit spectral radius need not have any Perron-Frobenius eigenvectors. 
The following example shows that, in general, the maximal Perron-Frobenius eigenvector $q_{\tau }$ of $\tau $ need not be a projection. See Section 4.6 of  \cite{BT}  for some related comments.

\begin{eg}
Let $\tau:\mathbb{M}_2 \to \mathbb{M}_2$ be a contractive CP map, defined by,
$$\tau(\begin{bmatrix}x_{11}&x_{12}\\x_{21}&x_{22}  \end{bmatrix})=\begin{bmatrix}x_{11}&0\\0&\frac{x_{11}}{2}\end{bmatrix}.$$ 
Here $q_\tau =\begin{bmatrix}1&0\\0&\frac{1}{2}\end{bmatrix}$ and $\mathcal{P}(\tau )=\mathbb{C}q_{\tau }.$
\end{eg}

Motivated by Proposition 4.6.5 in \cite{BT}, we have the following example, showing that any $q\geq 0$ with $\|q\|=1$, can become the identity of some peripheral Poisson boundary.
\begin{eg}\label{C}  Let $q\geq 0$ be an element in a von Neumann algebra $\mathcal{A}$ with $\|q\|=1.$  Let $\varphi $ be a normal state on $\mathcal{A}$ satisfying 
$\varphi(q)=1.$ For any scalar $\lambda $, with $0<\lambda <1,$ consider the contractive CP map $\tau$ defined by,
$$\tau(x)= \lambda x+(1-\lambda )\varphi(x)q.$$
Then it is easily seen that the peripheral Poisson boundary of $\tau$ is the one-dimensional space $\mathbb{C}q$, isomorphic to $\mathbb{C}$ with $q$ as its identity. 
\end{eg}

 A contractive CP map is said to be {\em peripherally automorphic} if the extended Choi-Effros product on the pereipheral Poisson boundary is same as the original product. Here we characterize peripherally automorphic maps on $\mathbb{M}_d$, in terms of their Choi-Kraus coefficients, extending a similar result proved for unital maps in \cite{BKT2}.

\begin{thm}
    Let $\tau:\mathbb{M}_d\to\mathbb{M}_d$ be a contractive CP map with $\tau(x)=\sum_{i=1}^n k_i^*xk_i$ for some $k_1, \ldots , k_n\in \mathbb{M}_d,  n\geq 1$. Then the following are equivalent:
    \item[(i)] $\tau$ is peripherally automorphic.
    \item[(ii)] For $\lambda \in \mathbb{T}$, $y\in E_{\lambda}(\tau)$ if and only if $yk_i=\lambda k_i y$ for all $1\leq i \leq n$. 
\end{thm}
\begin{proof}
  Assume (i). Let $\tilde{\tau}:\mathbb{M}_d\oplus \mathbb{C}\to \mathbb{M}_d\oplus \mathbb{C}$ be the unitization of $\tau .$ 
 We claim that $\tilde{\tau }$ is also peripherally automorphic. To see this note that, if $y\in E_\lambda(\tau), $ and $z\in E_\mu (\tau)$, where $\lambda, \mu \in \mathbb{T}$, the Choi-Effros product coincides with the natural product by the assumption on $\tau$. We know that $\mathcal{P}(\tilde{\tau})=\overline{\mbox{span}}~ ( \mathcal{P}(\tau)\oplus0 )\cup\{I\oplus1\}$, where $I=1_{\mathbb{M}_d},$ and $(y\oplus0) \circ (I\oplus1)=s.\lim_{n\to\infty}\frac{\tau^n(y)\oplus 0}{\lambda^n}=y\oplus 0=(y\oplus0) (I\oplus1)$. Similarly, we have $(I\oplus0)\circ (y\oplus0)=(I\oplus0) (y\oplus0)=y\oplus0$.
 Now we extend $\tilde{\tau }$ to a UCP map on $\mathbb{M}_{d+1}$ in a natural way by setting,
$$\tilde{\tilde{\tau}}\begin{bmatrix}
     x&x_{12}\\x_{21}&c
 \end{bmatrix}= \tilde{\tau}(x\oplus c)= 
 \begin{bmatrix}
     \tau(x)+c(1-\tau (1))&0\\0&c
 \end{bmatrix}
 $$
 where $ x\in \mathbb{M}_d, \ x_{12}\in \mathbb{M}_{d,1}, \ x_{21}\in \mathbb{M}_{1,d} \ \mbox{and} \ c\in \mathbb{C}$. Then, we have  $ \tilde{\tilde {\tau}}(z) =\sum_{i=1}^n m_i^*z m_i +\psi (z)$, for a suitable CP map $\psi $ and $m_i$'s of the form $$m_i=\begin{bmatrix}
     k_i & 0\\ 0&0
 \end{bmatrix}.$$ Based on the definition of $\tilde{\tilde{\tau}},$ we see that $\tilde{\tilde{\tau}}$ is also peripherally automorphic. Then, from Theorem 2.6 of  \cite{BKT2}, if $y\in E_{\lambda}(\tau ),$
 $$\begin{bmatrix}
     y& 0\\ 0&0
 \end{bmatrix} \begin{bmatrix}
     k_i & 0\\ 0&0
 \end{bmatrix}= \lambda \begin{bmatrix}
     k_i & 0\\ 0&0
 \end{bmatrix} \begin{bmatrix}
     y & 0\\ 0&0
 \end{bmatrix}, $$
 which yields (ii). 

 In the converse direction, (ii) implies (i) follows easily from the formula (\ref{strong formula}).

\end{proof}

An interesting fact is that if $\tau$ is peripherally automorphic, then $q_\tau =\lim _{n\to \infty}\tau ^n(I)$ is a projection in $\mathbb{M}_d$.  This is because $q^2_\tau=q_\tau\circ q_\tau=q_\tau$.

 \section{One parameter quantum dynamical semigroups}

Here we first recall the minimal dilation theorem for one-parameter
semigroups of completely positive maps on von Neumann algebras and
then we extend the peripheral Poisson boundary theory to this
setting.

\begin{defn}\label{QDS} 
A {\em quantum 
dynamical semigroup\/} is a triple $(\mathcal{H},
\mathcal{A}, \tau )$ where $\mathcal{H}$ is a complex, separable
Hilbert space and  $\mathcal{A}\subseteq \mathscr{B}(\mathcal{H})$
is a von Neumann algebra. Further, $\tau =\{\tau _t: t\in
\mathbb{R}_+\}$ is a family of linear maps on $\mathcal{A}$
satisfying,
\begin{itemize}
\item[1)] $\tau_t$ is completely positive for every $t$;
\item[2)] $\tau_s(\tau_t(x))=\tau_{s+t}(x), ~~\forall s,t\in \mathbb{R}_+, x\in
\mathcal{A}$;
\item[3)] $\tau_t(1) \leq 1$ , $\forall t\in \mathbb{R}_+$;
\item[4)] $x \mapsto \tau_t(x)$ is normal, $\forall t\in
\mathbb{R}_+$;
\item[5)]$t \mapsto \tau_t(x)$ is continuous in strong operator topology for every
$x\in \mathcal{A}.$
\end{itemize}

A quantum dynamical semigroup is called {\em unital/conservative/Markov \/}
when $\tau_t(1)=1$, $\forall t$.
 It is called an {\em $E$-semigroup\/}  if
$\tau _t$ is a $*$-endomorphism for every $t$ and an {\em  $E$-semigroup \/}
is called an $E_0$-semigroup if it is unital.
\end{defn}

We may abbreviate quantum dynamical semigroup to QDS and quantum Markov semigroup to QMS. We make use of the following elementary example to illustrate some of the concepts involved.
\begin{eg}\label{Schur}

Take $\mathcal{A}= \mathbb{M}_2$ as the von Neumann algebra.
Fix a complex number $c$ with Re$(c)\leq 0.$ Define $\tau=\{\tau_t\}, \forall t \in \mathbb{R}_+$ on $\mathcal{A}$ by, 
$$\tau_t(\begin{bmatrix}
          x_{11} & x_{12}\\
          x_{21}& x_{22}
      \end{bmatrix})=\begin{bmatrix}
          x_{11} &  \exp{(ct)}x_{12}\\
          \exp{(\bar{c}t)}x_{21}& x_{22}
      \end{bmatrix}$$
The condition Re$(c)\leq 0$ ensures that $\tau _t $ is a Schur product with a positive matrix and hence completely positive.  Then it is easy to see that  $\tau $  is a quantum Markov semigroup.
\end{eg}

\begin{thm} {\em (\cite{Bha96}, \cite{Bha99}, \cite{MS})}
Let $(\mathcal{H}, \mathcal{A}, \tau )$ be a quantum dynamical
semigroup. Then there exists an $E$-semigroup
$(\mathcal{K},\mathcal{B},\theta)$, satisfying the following
conditions:
\begin{itemize}
\item[(i)] $\mathcal{K}$ is a Hilbert space such that $\mathcal{H } \subseteq \mathcal{K}$.
\item[(ii)] $\mathcal{B} \subseteq \mathscr{B}(\mathcal{K})$ is a von Neumann algebra satisfying $p\mathcal{B}p= \mathcal{A}$ where $p$ is the projection of $\mathcal{K} $ onto $\mathcal{H}$.
\item[(iii)]  $\tau_t(x)=p\theta_t(x)p$ ; $ \forall x \in \mathcal{A}$.
\item[(iv)]  $\mathcal{K}=\overline{span}\{\theta_{r_1}(x_1)\theta_{r_2}(x_2)\ldots \theta_{r_n}(x_n)u: r_1 > r_2 > \cdots
 r_n \geq 0, x_i \in \mathcal{A}, u \in
\mathcal{H}\}.$
\item[(v)] $\mathcal{B}$ is the von Neumann algebra generated by
$\{ \theta _t(x): t\in \mathbb{R}_+, x\in \mathcal{A}\}.$

\end{itemize}
\end{thm}

The triple $(\mathcal{K}, \mathcal{B}, \theta)$ is called the {\em minimal dilation\/} of $(\mathcal{H}, \mathcal{A}, \tau) $. The unit of $\mathcal{B}$ is same as $1_{\mathcal{K}}.$ Such a dilation is unique up to unitary equivalence. It is true that if $\tau $ is a quantum Markov semigroup (i.e., $\tau $ is unital) then so is $\theta$ and $\theta_t(p)$ increases to the identity as $t$ increases to infinity.

 The concept of fixed points and peripheral eigenvectors in the continuous case is analogous to the discrete case, and here is the relevant definition. 
\begin{defn} Let $\tau $ be a quantum dynamical semigroup on a von Neumann algebra $\mathcal{A}.$ 

{(i)} The set of harmonic elements or {\em fixed points\/}  of $\tau $ is given by  
\begin{equation}\label{fixed}
F(\tau):=\{x\in\mathcal{A}: \tau_t(x)=x;~ \forall t \in \mathbb{R}_+\}. \end{equation}

{(ii)} The {\em peripheral eigenvectors\/} of $\tau $ are given by \begin{equation}\label{Peripheral} E_{a}(\tau):=\{x\in \mathcal{A}: \tau_t(x)=\exp{(iat)} x;~ \forall t \in \mathbb{R}_+ \}, \end{equation}  with $a \in \mathbb{R}$.  
\end{defn}

As in the discrete case, we consider 
\begin{eqnarray*} E(\tau)&=&  \mbox{span}\{x\in E_a( \tau) :  a\in \mathbb{R} \}.\\
     \mathcal{P}(\tau)&=&\overline{\mbox{span}}{\{x\in E(\tau)\}}.\end{eqnarray*}

The space $F(\tau)$ also has the `$\circ$' Choi-Effros product as
in the discrete case and $(F(\tau ), \circ )$ is a von Neumann
algebra, known as the {\em Poisson boundary\/}  of $\tau .$ Now we wish to
have peripheral Poisson boundary by extending this product to $\mathcal{P}(\tau ).$

\subsection {Quantum Markov semigroups} In this subsection, we develop the peripheral Poisson boundary theory for quantum Markov semigroups. So, we are considering a triple $(\mathcal{H}, \mathcal{A}, \tau )$, where $\mathcal{A}$ is a von Neumann algebra and $\tau =\{\tau _t: t\in \mathbb{R}_+\}$ is a quantum Markov semigroup in the sense of Definition \ref{QDS}.
Like in the discrete case,  it is possible to lift peripheral eigenvectors of a quantum
Markov semigroup to that of its minimal dilation.

 \begin{lem}\label{strong formula}
Let $(\mathcal{K}, \mathcal{B}, \theta )$ be the minimal dilation of
a quantum Markov semigroup $(\mathcal{H}, \mathcal{A}, \tau ).$ Fix
$a\in \mathbb{R}$.  For every $x \in E_a(\tau )$, there exists a
unique $\widehat{x} \in E_a(\theta )$ such that $$\widehat{x} = ~\mbox{s.} \lim_{t \to \infty} \exp (-iat) \theta_t(x).$$
Conversely, if ${y}\in E_a(\theta )$ for some $a\in \mathbb {R}$, then $x\in E_a(\tau ),$ where $x= p{y}p$.

 \end{lem}
 \begin{proof}
    We prove existence of $\widehat{x}$, for $x \in \mathcal{A}=p\mathcal{B}p$. Let $x_0=x$, $x_t=\theta_t(x)\exp(-iat)$.

We show $\theta_s(p)x_t\theta_s(p)=x_s$;  $\forall 0 \leq s \leq t$:

$$ \theta_s(p)x_t\theta_s(p)= \theta _s(p) \theta _t(x)\theta _s(p)\exp (-iat)= \theta _s(p\theta _{t-s}(x)p)\exp (-iat) $$  $$= \theta _s(\tau _{t-s}(x))\exp(-iat) = \theta _s(x)\exp (-ias)=x_s.$$
Since $\theta_t (p)$ is a family of projections that increases to identity,
for $h \in \mathcal{K}$ and $\epsilon > 0$, we can choose $t_0$ such
that $\|h-\theta_{t_0}(p) h\|< \epsilon$. Taking $h_0=\theta_{t_0}(p)h$,
for $ t_0<s<t$, $x_sh_0$  and $x_t h_0 - x_s h_0$ are
mutually orthogonal, therefore $\{\|x_t h_0\|^2\}_{t\geq t_0}$ is an increasing function bounded by $\|x\|^2\|h_0\|^2$.
Also, $$\|(x_t-x_{s})h_0\|^2 =\|x_t h_0\|^2-\|x_{s} h_0\|^2.$$
Therefore, $x_t h_0$ converges as $t\to \infty$. As this is true for every $\epsilon >0$,  $x_t h$ is also convergent as $t$ tends to infinity. 
Taking $\widehat{x}h:=\lim_{t \to \infty}x_t h$ and since $\|x_t
h\|\leq\|x\|\|h\|, \forall t$,  $\widehat{x}$ defines a bounded
operator in $\mathcal{B}$. It is clear that
$\theta_t(p)\widehat{x}\theta_t(p)=x_t$.
$$\theta_s(\widehat{x})=\theta_s(s.\lim_{t \to \infty}x_t)=\theta_s(s.\lim_{t \to
\infty}\frac{\theta_t(x)}{\exp(iat)})=\lim_{t \to
\infty}\exp(ias)\frac{\theta_{t+s}(x)}{\exp ia(t+s)}=\exp(ias)
\widehat{x},$$
    and $p\widehat{x}p=x$. To prove the uniqueness, suppose there exists $z \in \mathcal{B}$,
such that $\theta_t(z)=\exp(iat) z$, and $pzp=x.$
     $$z= s.\lim_{t \to \infty} \theta_t(p)z\theta_t(p)=s.\lim_{t \to \infty}\exp(-iat) \theta_t(pzp)=s.\lim_{t \to \infty}\exp(-iat) \theta_t(x)=\widehat{x}.$$

To see the second part, let $y\in \mathcal{B}$, such that
$\theta_s(y)=\exp(ias) y$, for some $a \in
\mathbb{R}$. Take $x=pyp$. We get
$$\tau_s(x)=\tau_s(pyp)=p\theta_s(pyp)p=p\theta_s(p)\theta_s(y)\theta_s(p)p=p\theta_s(p)\theta_s(y)\theta_s(p)p=\exp(ias)
x.$$ Hence $x\in E_a(\tau ).$

 \end{proof}

Now we are ready to prove the main theorem of this section. It allows us to define the peripheral Poisson boundary for quantum Markov semigroups.
\begin{thm}\label{Main 4}
     Let $\mathcal{A} \subseteq \mathscr{B}(\mathcal{H})$ be a von Neumann algebra, let $\tau=\{\tau_t\}:\mathcal{A}\to \mathcal{A}$ be a quantum Markov semigroup. Let $(\mathcal{K},\mathcal{B},\theta)$ be it's minimal dilation and let $p $ be the projection of $\mathcal{K}$ onto $\mathcal{H}$. Then the compression map $z \mapsto pzp$, restricted to $\mathcal{P}(\theta)$ is a completely positive map mapping the operator space $\mathcal{P}(\theta)$ completely isometrically and bijectively to $\mathcal{P}(\tau)$.
\end{thm}
\begin{proof}
Let us denote the compression map $z\mapsto pzp$ restricted to $\mathcal{P}(\theta )$ by $T.$
By the previous Lemma \ref{strong formula},  we have a map $x\mapsto \widehat{x}$ from $E(\tau )$ to $E(\theta )$, whose inverse is the compression map $T$. So $T$ is a bijection. It is clearly completely positive. Proving the isometry property requires some estimates.

Let $x \in E(\tau)$.  Suppose $x=\sum_{j=1}^{d}x_j$, where each $x_j \in E_{a_j}(\tau)$ for some $a_1, \ldots a_d \in \mathbb{R} . $ Now take  $y_j=\lim_{s \to \infty}\exp (-ia_js)\theta_s(x_j)$
and  $y=\sum_{j=1}^{d}y_j$. Fix $\epsilon >0.$
For fixed $h \in \mathcal{K}$ with norm equal to 1,
there exists $s_0>0$ such that $\forall s \geq s_0,$ $$||yh-\sum_{j=1}^{d}\exp (-ia_js )\theta_s(x_j)h||<\frac{\epsilon}{2}.$$
By a standard result (see \cite{BKT1}, Lemma 2.2), it is possible to
choose $m\in \mathbb{N}, m>s_0,$ such that $$\|\exp (ia_jm)-1\|\leq \frac{\epsilon}{2(\sum_{j=1}^{d}||x_j||)},~~ \forall 1\leq j\leq d.$$ By contractivity of endomorphisms, $||\theta_m(x_j)h||\leq ||x_j||$.  Then by triangle inequality,
\begin{eqnarray*}
    \|yh\|&\leq & \frac{\epsilon}{2}+
     \| \sum_{j=1}^{d}[\exp(-ia_j m)  -1]\theta _m(x_j)h\|
    +\|\sum _{j=1}^d\theta _m(x_j)h\|\\
    &\leq & \epsilon + \|\theta _m(\sum _{j=1}^dx_j)h\|\\
    &\leq & \epsilon +\|xh\|.      
\end{eqnarray*}
As this is true for every unit vector $h$ and $\epsilon >0$, 
we get $\|y\| \leq \|x \|.$
 Since $p$ is a projection and $x=pyp$, we get $\|x\|\leq \|y\|.$
    Therefore, $\|x\|= \|y\|$ and hence $T$   from $\mathcal{P}(\theta)$ onto $\mathcal{P}(\tau)$ is an isometry.
\end{proof}

Now we can define the extended Choi-Effros product on $\mathcal{P}(\tau )$, setting $x\circ y =p\widehat{x}.\widehat{y}p$, or equivalently $x\circ y= T (T^{-1}(x).T^{-1}(y))$. This new product makes  $\mathcal{P}(\tau )$  a  unital $C^*$-algebra as $\mathcal{P}(\theta )$ is a unital $C^*$-algebra with its original multiplication. This leads to the formal definition:

\begin{defn} Let $(\mathcal{H}, \mathcal{A}, \tau )$ be a quantum Markov semigroup. Then the $C^*$-algebra $(\mathcal{P}(\tau ), \circ)$ is called the {\em peripheral Poisson boundary\/}  of $\tau .$
\end{defn}

Here is one of the main properties of the peripheral Poisson boundary.
\begin{thm}
 Let $\tau=\{\tau_t :t\in \mathbb{R}_+\}$ be a quantum Markov semigroup on a von Neumann algebra $\mathcal{A}$.  Then $x\mapsto \tau_t(x)$ is an automorphism of the peripheral Poisson boundary $(\mathcal{P}(\tau), \circ )$. Consequently, restricted to $(\mathcal{P}(\tau ), \circ)$, $\{\tau _t: t\in \mathbb{R}_+\}$ is a semigroup of automorphisms.
\end{thm}
\begin{proof}
    Consider the map $x\to\theta_t(x)$ on the peripheral Poisson boundary of the quantum Markov semigroup $\theta=\{\theta_t\}$.
Clearly, $\theta_t:E(\theta)\to E(\theta)$ is a $*$-endomorphism. As $\theta_t(x)=\exp(iat) x; \ \forall x \in E_a(\theta)$, every $x\in E(\theta )$ is in the range of $\theta $ and is thus surjective.
 We prove $\theta_t$ is isometric, therefore one-one in $E(\theta)$; consequently will be an isometry for $\mathcal{P}(\theta)$.
Since $\theta_t$ is a $*$-endomorphism $||{x}|| \geq ||\theta_t(x)||
\geq ||\theta_{nt}(x)||, \forall t\in \mathbb{R}_+, n\in
\mathbb{N}$. Consider $x=\sum_{j=1}^{d}x_j$,
with $x_j \in E_{a_j}(\theta)$; for distinct $a_1, a_2, \ldots ,a_d
\in \mathbb{R}$.

For $\epsilon \geq 0$, there exists $n$ such that $$|\exp(ia_j
nt)-1|<\frac{\epsilon}{2(\sum_{j=1}^{d}||x_j||)}.$$
So,
$$||\theta_{nt}(x)-x||=||\sum_{j=1}^{d}(\exp(ia_j
nt)-1)x_j||<\epsilon \Longrightarrow ||\theta_t(x)||=||x||.$$
Then using the compression map $z\mapsto pzp$  from $E(\theta)$ to $E(\tau )$, we see that $\tau $  restricted to $\mathcal{P}(\tau )$ is a semigroup of automorphisms of $(\mathcal{P}(\tau ) , \circ)$.
\end{proof}

As in the discrete case, we have the following formula for the extended Choi-Effros product.

\begin{thm} \label{continuous limit formula}
Let $(\mathcal{H}, \mathcal{A}, \tau )$ be a quantum Markov semigroup. Then for $x\in E_a(\tau ), y\in E_b(\tau ),$ with $a, \ b \in \mathbb{R}$;  
\begin{equation}\label{formula}
x\circ y= ~\mbox{s.}\lim_{t \to \infty} \frac{\tau_t(xy)}{\exp(iat)\exp(ibt)}.\end{equation}
\end{thm}
\begin{proof}
With notations as above,
$$x\circ y= p\widehat{x}.\widehat{y}p
=p(~\mbox{s.}\lim_{t \to \infty}\frac{\theta_t(xy)}{\exp(iat)\exp(ibt)})p=~\mbox{s.}
   \lim_{t \to \infty}\frac{\tau_t(xy)}{\exp(iat)\exp(ibt)},$$
   by the dilation property.
\end{proof}

In the following, let $(\mathcal{H}, \mathcal{A}, \tau )$ be a quantum Markov semigroup.

\begin{cor}
    If the von Neumann algebra $\mathcal{A}$ is abelian, then the peripheral Poisson boundary of the QMS $\{\tau_t:  t\in \mathbb{R}_+\}$ on $\mathcal{A}$ is also abelian.
\end{cor}
\begin{proof} This is clear from the formula for the extended Choi-Effros product in
(\ref{formula}). 
\end{proof}

\begin{cor}
   For $a,b \in \mathbb{R}$,  if $x \in E_a(\tau)$, $y \in E_b(\tau)$, then $x\circ y \in E_{a+b}(\tau)$.
\end{cor}
\begin{proof} Using  formula (\ref{formula}),

 \begin{eqnarray*} \tau_s(x \circ y)&=&~\mbox{s.} \lim_{t \to \infty}\frac{\tau_s\tau_t(xy)}{\exp ia(s+t)\exp ib(s+t)} \exp(ias)\exp(ibs)\\&=&\exp [i (a+b)s ](x\circ y).\end{eqnarray*}
Hence $x\circ y\in E_{a+b}(\tau ).$ 
\end{proof}

\begin{eg}
 In Example \ref{Schur}, the matrix units $E_{11}, E_{12}, E_{21}, E_{22}$
are eigenvectors of $\tau _t$ with eigenvalues $1, \exp{(ct)}, \exp{\bar{c}t}, 1$ respectively.  If Re$(c)<0$, then $\mathcal{P}(\tau )$ consists of only diagonal matrices and the peripheral Poisson boundary is just the $C^*$-algebra $\mathbb{C}^2$ of diagonal matrices.
On the other hand, if Re$(c)=0$, $\tau $ is a semigroup of automorphisms. Here, all the eigenvalues are peripheral. The peripheral Poisson boundary is the whole of $M_2(\mathbb{C})$, with the usual matrix multiplication.

Observe that in general if $\tau $ is a semigroup of automorphisms of a finite-dimensional von Neumann algebra, then the peripheral Poisson boundary is the algebra itself. 
\end{eg}

The following is a continuous version of Example \ref{C}.

\begin{eg}Let $q\geq 0$ be an element in a von Neumann algebra $\mathcal{A}$ with $\|q\|=1.$  Let $\varphi $ be a normal state on $\mathcal{A}$ satisfying 
$\varphi(q)=1.$ Now, we define the QDS $\tau=\{\tau_t\}$ as follows: For $a>0$,
$$\tau _t(x) =e^{-at}x+(1-e^{-at})\varphi(x)q, ~~x\in \mathcal{A}.$$
Clearly, $q$ is a fixed point for $\tau=\{\tau_t\}$. In  fact, the peripheral Poisson boundary of $\tau$ is the one-dimensional space $\mathbb{C}q$, isomorphic to $\mathbb{C}$.
\end{eg}

\begin{eg}
  Consider the von Neumann algebra $\mathcal{A}=L^\infty(\mathbb{T})$, and the unital QDS $\tau=\{\tau_t\}$ defined as: $$\tau_t(f)(z)=f(\exp{(it)}z), \ \forall f\in L^\infty(\mathbb{T}), \ z\in \mathbb{T}.$$
  Clearly,  $z^a \in E_a(\tau), \ \forall a \in \mathbb{R}$. By Stone-Weierstrass theorem, this implies $C(\mathbb{T})\subseteq \mathcal{P}(\tau)$.
Computing the extended Choi-Effros product using the formula as above, we have $z^a \circ z^b = z^{a+b}, \ \forall a, b \in \mathbb{R}$. This is not surprising, as the extended Choi-Effros product is the same as the original product for any endomorphism semigroup. The peripheral Poisson boundary is not the whole of $\mathcal{A}$ only because in our definition of the boundary we have considered closure only in norm and not in any weaker topologies. 
\end{eg}

\begin{thm} Let $\tau $ be a quantum Markov semigroup. For $a \in \mathbb{R}$, take $$I_a(\tau)=~\overline{\mbox{span}}~\{x^*\circ y: x,y \in E_a(\tau)\}.$$ 
   \item[(i)]  $(I_a(\tau), \circ)$ is a closed two-sided ideal of the  Poisson boundary  $(F(\tau), \circ)$.
   \item[(ii)] The mapping $x \to x^*$ from $E_a(\tau)$ onto $E_{-a} (\tau)$ is an anti-isomorphism.
\end{thm}
\begin{proof}
   The proof follows as in Theorem 2.9, of \cite{BKT1} replacing $\theta$ by the quantum Markov semigroup $\theta=\{\theta_t\}$.
\end{proof}

\begin{cor}\label{dimension}
   Fix $a \in \mathbb{R}$. (i) If there exists $v_1, v_2 \in E_a(\tau)$, such that $v_1^*\circ v_2 = 1$, then
   $$E_a(\tau)=\{ x\circ v_2: x \in F(\tau)\}=\{x\circ v_1: x \in F(\tau)\}.$$
(ii) If there exists $v_1, v_2 \in E_a(\tau)$, such that $v_1 \circ v_2^*=1$, then $$E_a(\tau)=\{ v_2 \circ x: x \in F(\tau)\}=\{v_1 \circ x: x \in F(\tau)\}.$$
In either of these cases,  $\dim(E_a(\tau) )=\dim(F(\tau))$.
\end{cor}
\begin{proof} For $(i)$, consider $v_1, v_2\in E_a(\tau )$ such that $v_1^* \circ v_2=1.$ First, we show that $\{ x\circ v_2: x \in F(\tau)\}\subseteq E_a(\tau)$. If $x \in F(\tau)$; $x\circ v_2 \in E_a(\tau)$, i.e, $\tau_s(x \circ
v_2)=\exp(ias) (x\circ v_2), \ \forall s \in \mathbb{R}$. Next, we want to show that $E_a(\tau)\subseteq \{ x \circ v_2: x \in F(\tau)\}.$ For $y \in E_a(\tau)$, $y\circ v_1^* \in F(\tau)$, $(y\circ v_1^*)\circ v_2=y$.
It follows similarly for $(ii)$.

Now, we prove that $\dim(E_a(\tau))=\dim(F(\tau))$.
If $v_1^*\circ v_2=1$, and $x_1, x_2, ... , x_n \in F(\tau)$ are linearly independent, then $v_2\circ x_1, ... , v_2\circ x_n$ are linearly independent on $E_a(\tau)$.
Conversely, if $y_1, y_2,... y_n \in E_a(\tau)$ are linearly
independent, $y_1 \circ v_1^*, ..., y_n\circ v_1^*$ are linearly
independent in $F(\tau)$. Hence, the equality in dimensions.
\end{proof}

\begin{cor}
    If $F(\tau)$ is 1 dimensional and $E_a(\tau)$ is non-trivial, then $E_a(\tau)$ is also one-dimensional.
\end{cor}% \textcolor{blue}{Where is the proof of this Corollary?}
\begin{proof}
 For $\tau$, QMS,  $F(\tau)$ is one-dimensional, implies $F(\tau)=\mathbb{C} 1.$ If $x\in E_a(\tau)$, then $ x^* \circ x \in F(\tau)$.  Without loss of generality, let  $x^*\circ x=1$. Then by Corollary \ref{dimension}, $\mbox{dim}(E_a(\tau))=\mbox{dim}(F(\tau))$.
\end{proof}

\begin{defn}
A quantum Markov semigroup $\tau=\{\tau_t\}$ is said to be {\em  peripherally automorphic\/} if $x
\circ y=xy, ~\forall x,y \in \mathcal{P}(\tau)$.
\end{defn}

\begin{defn}
    The {\em multiplicative domain\/}  of a quantum Markov semigroup $\tau=\{\tau_t\}$ is defined as

$\mathcal{M}(\tau ):=\{x \in \mathcal{A}: \tau_t(xy)=\tau_t(x)\tau_t(y), \tau_t(yx)=\tau_t(y)\tau_t(x); \forall y \in \mathbb{M}_d , t \in \mathbb{R}_+\}$.
\end{defn}

\begin{thm}
Let $(\mathcal{H}, \mathcal{A}, \tau )$ be a quantum Markov semigroup.  Then the following are equivalent:
\begin{itemize}
\item[(i)] $\mathcal{P}(\tau) \subseteq \mathcal{M}(\tau )$; \item[(ii)] $\tau$ is peripherally automorphic;\item[(iii)] For every $a \in \mathbb{R}, x \in E_a(\tau) \Longrightarrow x^*x\in F(\tau ).$
\end{itemize}
\end{thm}
\begin{proof}
   Assuming (i) to be true. Suppose $x\in E_{a}(\tau )$ and $y\in E_b(\tau )$, for some $a,b \in \mathbb{R}$. Then
$$x\circ y= s.\lim_{t \to \infty}
\frac{\tau_t(xy)}{\exp(iat)\exp(ibt)}=s.\lim_{t \to \infty}
\frac{\tau_t(x)\tau_t(y)}{\exp(iat)\exp(ibt)}= xy.$$
So, (ii) follows. 

Assuming (ii) to be true.
 Note that $x^*\circ x= x^*x$, due to $\tau$ being peripherally automorphic. Then combining it with $\tau_t(x^*\circ x)=x^*\circ x; \forall x\in \mathcal{P}(\tau)$ gives us (iii).

Now assume (iii). Some simple calculations show that
$$\mathcal{M} (\tau )=\{x \in \mathcal{A}: \tau_t(xx^*)=\tau_t(x)\tau_t(x^*),
\tau_t(x^*x)=\tau_t(x^*)\tau_t(x); \forall , t \in \mathbb{R}_+\}.$$
Hence, (iii) implies (i).    
\end{proof}

\subsection{Quantum dynamical semigroups} Here we extend the theory of peripheral Poisson boundary to general quantum dynamical semigroups on von Neumann algebras, which, by our assumption, are contractive but not necessarily unital. This extension can be achieved as in the discrete case by a unitization trick.

Suppose $(\mathcal{H}, \mathcal{A}, \tau)$ is a quantum dynamical semigroup, where it is assumed that $\tau$ is not unital. Now take
$\widetilde{\mathcal{H}}= \mathcal{H}\oplus \mathbb{C}$, $\widetilde{\mathcal{A}}=\mathcal{A}\oplus \mathbb{C}.$ For $t\in \mathbb{R}_+$, define $\widetilde{\tau _t}:\widetilde{\mathcal{A}}\to \widetilde{\mathcal{A}},$ by
$$\widetilde{\tau _t}(x\oplus c)= (\tau _t(x)+ c(1-\tau _t(1))\oplus c, ~~x\oplus c\in \mathcal{A}\oplus \mathbb{C}. $$
It is easily verified that $(\widetilde{\mathcal{H}}, \widetilde{\mathcal{A}}, \widetilde{\tau })$ is a quantum Markov semigroup. The minimal dilation of $(\mathcal{H}, \mathcal{A}, \tau )$ can be constructed using the minimal dilation of 
$(\widetilde{\mathcal{H}}, \widetilde{\mathcal{A}}, \widetilde{\tau }).$ See \cite{Bha99} for details. Now we can mimic the discrete case. We just state the results without proofs.

\begin{thm} Let
$(\mathcal{H}, \mathcal{A}, \tau )$ be a quantum dynamical semigroup. Let $(\mathcal{K}, \mathcal{B}, \theta ) $ be its minimal dilation and $p$ be the projection of $\mathcal{K}$ onto $\mathcal{H}.$ For $x\in E_a(\tau ), a\in \mathbb{R}$ there exists a unique $\widehat{x}\in E_a(\theta )$ such that $x=p\widehat{x}p.$ The map $x\mapsto \widehat{x}$ extends to a complete isometric isomorphism of the operator space $E(\tau )$ to $E(\theta ).$
\end{thm}

Defining $x\circ y= p\widehat{x}.\widehat{y}p,$ makes $\mathcal{P}(\tau )$ a $C^*$-algebra, called the peripheral Poisson boundary of $\tau .$ Whenever it is non-trivial, it is a unital $C^*$-algebra with its unit given by $\mbox{s}.\lim _{t\to \infty}\tau _t(1).$ Furthermore, on the peripheral Poisson boundary $(\mathcal{P}(\tau ), \circ )$, the semigroup $\tau $ is a semigroup of automorphisms. 

\section{UCP maps on $C^*$-algebras}

We want to show that the concept of peripheral Poisson boundary does not extend to the $C^*$-algebra setting. Firstly in certain examples the lifting of peripheral eigenvectors is not possible within the $C^*$-algebraic dilation, however they do exist by considering the associated von Neumann algebra. For these examples, it is possible to define the peripheral Poisson boundary by considering the larger dilation. In other examples, it is impossible to have any $C^*$-algebraic structure on the operator space obtained by the closed linear span of peripheral eigenvectors. 

\begin{eg}  Consider the Hilbert space $\mathcal{H}=l^2(\mathbb{Z}_+)$, with 
standard orthonormal basis $\{ e_n : n\in\mathbb{Z}_+\}$.
Let $S$ be the unilateral shift on $\mathcal{H}$, defined by $Se_n= e_{n+1}, ~n\in \mathbb{Z}_+$ and extended linearly to an isometry.
Let $\tau$ be defined on $\mathcal{B}(\mathcal{H})$, as $\tau (X)= S^*XS.$ Clearly, $\tau $ is a normal UCP map.  Its dilation can be easily written down. Let $\mathcal{K}= l^2(\mathbb{Z})$ be a Hilbert space with the standard orthonormal basis $\{e_n:n\in \mathbb{Z}\}.$ Clearly $\mathcal{H}$ is a subspace of $\mathcal{K}$ in a natural way. Let $U$ be the bilateral shift on $\mathcal{K}$, defined by $Ue_n=e_{n+1}, ~n\in \mathbb{Z}$, and extended linearly to an unitary. Then $\theta(Z)=U^*ZU$, is an automorphism of $\mathscr{B}(\mathcal{K})$ and it is pretty obvious that
$(\mathcal{K}, \mathscr{B}(\mathcal{K}), \theta )$ is the minimal dilation of $(\mathcal{H}, \mathscr{B}(\mathcal{H}), \tau)$ in the von Neumann
algebra setting.

However, if we consider the $C^*$-algebra setting, the picture is different. Decompose $\mathcal{K}$ as $\mathcal{K}=\mathcal{K}_1\oplus \mathcal{K}_2$,
where $\mathcal{K}_2=\mathcal{H}$ and $\mathcal{K}_1= (\mathcal{H})^\perp .$ Now any operator $Z$ on $\mathcal{K}$ decomposes as a $2\times 2$ block operator matrix:
$$Z=\left[\begin{array}{cc}Z_{11}& Z_{12}\\Z_{21}&Z_{22}\end{array}\right],$$
where $Z_{ij}\in \mathscr{B}(\mathcal{K}_j, \mathcal{K}_i),~~1\leq i, j\leq 2.$ Take 
$$\mathcal{A}= \{Z: Z\in \mathscr{B}(\mathcal(K)): Z_{11}, Z_{12}, Z_{21}~\mbox{are compact},~ Z_{22}\in \mathscr{B}(\mathcal{H})\}.$$
 Then $\mathcal{A}$ is a $C^*$-algebra left invariant by $\theta .$ Let $\psi $ be the restriction of $\theta $ to $\mathcal{A}.$ Now it is not hard to see that $(\mathcal{K}, \mathcal{A}, \psi )$ is the dilation of $(\mathcal{H}, \mathscr{B}(\mathcal{H}), \tau )$ in the $C^*$-algebra setting. 

The fixed points of $\tau $ are precisely the Toeplitz operators, and the peripheral eigenvectors of $\tau $ are known as $\lambda $-Toeplitz operators (see \cite{BKT1}, remarks before Theorem 3.3). In particular, the isometry $S$ itself is a fixed point of $\tau $. It lifts to $U$ in the von Neumann algebra setting. However, $U$ is not in the algebra $\mathcal{A}.$ This means that $S$ has no lift in the $C^*$-algebraic dilation. Nevertheless, we can define a Choi-Effross product and make $\mathcal{P}(\tau )$ a unital $C^*$-algebra, by considering the von Neumann algebraic dilation.
\end{eg}

\begin{eg}\label{Example2}
Take $\mathcal{H}=l^2(\mathbb{Z}_+ )$ as the Hilbert space, with standard orthonormal basis $\{e_n:n\in \mathbb{Z}_+\}$.  Let $\tau :\mathscr{B}(\mathcal{H})\to \mathscr{B}(\mathcal{H})$ be the  UCP map  defined by
   $$\tau(\begin{bmatrix}
x_{00}&x_{01}&x_{02}&x_{03}&x_{04}&..&..\\x_{10}&x_{11}&x_{12}&x_{13}&x_{14}&..&..\\x_{20}&x_{21}&x_{22}&x_{23}&x_{24}&..&..\\x_{30}&x_{31}&x_{32}&x_{33}&x_{34}&..&..\\x_{40}&x_{41}&x_{42}&x_{43}&x_{44}&..&..\\:&:&:&:&:&:&:
   \end{bmatrix})=\begin{bmatrix}
       x_{00}&x_{01}&0&0&0&..&..\\x_{10}&x_{11}&0&0&0&..&..\\0&0&x_{00}&0&0&..&..\\0&0&0&x_{11}&0&..&..\\0&0&0&0&x_{22}&..&..\\:&:&:&:&:&:&:
   \end{bmatrix}.$$
The peripheral eigenvectors of $\tau $ are all fixed points and the fixed point space is spanned by $E_{01}$, $E_{10}$, $P_0$, and $P_1$, where $E_{ij}=|e_i\rangle \langle e_j|$ for all $i,j$, and $P_0, P_1$ are projections onto $\overline{\mbox{span}}\{e_j: j ~\mbox{is even}\}$ and $\overline{\mbox{span}}\{e_j: j~\mbox{is odd}\}$ respectively. The extended Choi-Effros product makes these eigenvectors as new matrix units satisfying $E_{01}\circ E_{10}=P_0$ and $E_{10}\circ E_{01}=P_1$. In particular, the peripheral Poisson boundary is isomorphic to the $2\times 2$ matrix algebra $\mathbb{M} _2.$ 

Now consider the $C^*$-algebra $\mathcal{A}:= \{ cI+K: c\in \mathbb{C}, K\in \mathscr{K}\}$ of compact perturbations of scalars. Observe that $\tau $ leaves $\mathcal{A}$ invariant. Let $\tau _0 $ be the restriction of $\tau$ to $\mathcal{A}$. We look at the triple $(\mathcal{H}, \mathcal{A}, \tau _0) $. 
The peripheral eigenspace of $\tau _0$ is spanned by only three vectors
$E_{01}, E_{10}$ and $ I$, as $P_0$ and  $P_1$ are individually not in $\mathcal{A}.$
We wish to show that the operator system $E(\tau _0)=~ \mbox{span}\{E_{01}, E_{10}, I\}$ does not admit any $C^*$-algebra structure (retaining same operator system structure). 
Now it is convenient to have the following general definition.

\begin{defn}
 Let $E$ be an operator system. Then $A\in E$ is said to be {\em pure\/} if $A\geq 0$, and for all $B$ in $E$ with $0 \leq B \leq A$, satisfy
$B=qA$ for some scalar $q\geq 0$.    
\end{defn}

Note that the space $E(\tau _0)$ is $3$- dimensional, and so the only possibility for it to be a $C^*$-algebra is  that it is isomorphic to $\mathbb{C}^3$.
Clearly in $\mathbb{C}^3,$ the only pure vectors of unit norm are standard basis vectors. So there are exactly three of them. 

In $E(\tau _0)$, if $A=aI+bE_{01}+cE_{10}$, then
$A^*=\overline{a}I+\overline{b}E_{10}+\overline{c}E_{01}$ and
$\|A\|=\|\tilde{A}\|$, where  $$\tilde{A}:=\begin{bmatrix}
       a&b\\c&a
   \end{bmatrix} \in \mathbb{M}_2.$$
It is easy to see that the set of unit norm pure vectors of
$E_1(\tau_0)$ is given by,
$$\{ \frac{1}{2}(I+\exp^{i\theta}E_{01}+\exp^{-i\theta}E_{10}): \theta
\in [0,2\pi)\}.$$ As this set is infinite, the space  $E(\tau _0)$ cannot have a $C^*$-algebra structure making it isomorphic to $\mathbb{C}^3$.
\end{eg}

Here are two continuous time, that is, one parameter semigroup examples.

\begin{eg} 
    Consider the unital $C^*$-algebra $\mathcal{A}=C_c(\mathbb{R}_+)$, of continuous functions on $\mathbb{R}_+$ that have a limit as the variable goes to infinity,  with the unital quantum dynamical semigroup, $\tau=\{\tau_t: t\in \mathbb{R}_+\}$ defined as the shift $\tau_t (f)(x)=f(x+t), \ x\in \mathbb{R}_+$. The identity function on $\mathbb{R}_+$ is a peripheral eigenvector. It turns out that there is no lift for this eigenvector in the dilation of $\tau$. The minimal dilation of $\tau$ is the  shift $\theta _t(f)(x)= f(x+t), ~~x\in \mathbb{R},$
    on the non-unital $C^*$-algebra:
    $$\{ g\in C(\mathbb{R}): \lim _{x\to -\infty}g(x)=0, \lim _{x\to \infty}g(x)~\mbox{exists} \},$$ where $C(\mathbb{R})$ denotes the space of continuous functions on $\mathbb{R.}$
      The natural candidate for the lift would have been the identity function on $\mathbb{R}.$ But it is not an element of this algebra.
\end{eg}

\begin{eg}
Consider the von Neumann algebra $\mathcal{A}=\mathbb{M}_2 \oplus L^\infty (-\infty,\infty), $ and the QDS $\tau=\{\tau _t:t\in \mathbb{R}_+\}$ defined by, $$\tau_t(\begin{bmatrix}
        x_{11}&x_{12}\\x_{21}&x_{22}
    \end{bmatrix}\oplus f)=\begin{bmatrix}
        x_{11}&x_{12}\\x_{21}&x_{22}
    \end{bmatrix}\oplus \tilde{f},~\mbox{ where}~ \tilde{f}(x)=\begin{cases}
    f(x+s) & \quad x \in (-\infty,-s); \\
    x_{11} & \quad x \in [-s,0);\\
        x_{22} & \quad x \in [0,s); \\
        f(x-s) & \quad x \in [s,\infty).
    \end{cases}$$ The peripheral eigenvectors of $\tau$ are the  fixed points and they are given by: 
    $$\{\begin{bmatrix}
x_{11}&x_{12}\\x_{21}&x_{22}
    \end{bmatrix}\oplus (x_{11}1_{(-\infty,0)}+x_{22}1_{[0,\infty)})  : x_{11}, x_{12}, x_{21}, x_{22}\in \mathbb{C}\},   $$
where $1$ indicates the indicator function on the relevant interval.    
    
Now, consider the $C^*$-algebra $\mathcal{A}_0\subseteq \mathcal{A}$ given by $\mathcal{A}_0=\mathbb{M}_2 \oplus (\mathbb{C}1+\mathcal{K}_{(-\infty,\infty)} )$, where $K_{(-\infty,\infty)}$ is the $C^*$-algebra generated by compactly supported bounded measurable functions on $(-\infty , \infty ).$  Clearly $\tau $ leaves $\mathcal{A}_0$ invariant. Restrict $\tau$ to $\mathcal{A}_0$. The peripheral eigenvectors for $\tau$ restricted to $\mathcal{A}_0$ are spanned by $$\begin{bmatrix}
        1&0\\0&1
    \end{bmatrix}\oplus 1, \ \begin{bmatrix}
0&1\\0&0
    \end{bmatrix}\oplus 0, \ \mbox{and} \  \begin{bmatrix}
0&0\\1&0
    \end{bmatrix}\oplus 0.$$  It is impossible to induce a product on this operator space,  $\mathcal{P}(\tau )$  to make it a $C^*$-algebra for the same reason as in Example \ref{Example2}.
    
\end{eg}

In earlier sections, we have seen that for a normal unital completely positive map on a von Neumann algebra, using the extended Choi-Effros product we can induce a $C^*$- algebra structure on the operator system $\mathcal{P}(\tau )$. If the complete positivity condition is dropped and just positivity is kept, one may not be able to induce any such product.
We illustrate this with the following example.

\begin{eg}
Let $\tau: \mathbb{M}_2 \to \mathbb{M}_2$ be a unital positive map, defined as:
$$\tau(\begin{bmatrix}
        x_{11}&x_{12}\\x_{21}&x_{22}
    \end{bmatrix})=\begin{bmatrix}
        \frac{x_{11}+x_{22}}{2}&x_{21}\\x_{12}&\frac{x_{11}+x_{22}}{2}
    \end{bmatrix}.$$ The peripheral eigenvalues are $1$ and $-1$, with the corresponding eigenspaces:
    $$E_1(\tau )= ~\mbox{span}~\{ \begin{bmatrix}
        1&0\\0&1
    \end{bmatrix}, \begin{bmatrix}
        0&1\\1&0
    \end{bmatrix}\}; ~~E_{-1}(\tau )= \mbox{span}~\{ \begin{bmatrix} 0&-1\\1&0 \end{bmatrix}\}.$$
Clearly, if $\mathcal{P}(\tau)$ is a $C^*$-algebra, it can only be isomorphic to $\mathbb{C}^3$, but as an operator system, it is not isomorphic to $\mathbb{C}^3.$  (One may refer to Example \ref{Example2} for possible arguments).     
\end{eg}

Finally, here is an example to show that it may not be possible to extend the theory of peripheral Poisson boundary to non-contractive CP maps. 

\begin{eg}  Take the von Neumann algebra as $\mathcal{A}=\mathbb{M}_2$ and the CP map $\tau$ defined as
   $$\tau(\begin{bmatrix}
        x_{11}&x_{12}\\x_{21}&x_{22}
    \end{bmatrix})=\begin{bmatrix}
        2x_{11}&-x_{12}\\-x_{21}&2x_{22}
    \end{bmatrix}.$$ 

    Clearly the only peripheral eigenvalue is -1, and all non-trivial matrices with zeros on diagonals are eigenvectors.  So $\mathcal{P}(\tau )$ is spanned by matrix units $E_{12}$ and $E_{21}.$ It is impossible to induce a product on this operator space to make it a $C^*$-algebra.
    To see this notice, if there exists a product, say $\circ$, then $\mathcal{P}(\tau)\cong  \mathbb{C}^2$. 
  The operator space $\mathbb{C}^2$ has two self adjoint elements $e_1, e_2$ satisfying $\|e_1\|=\|e_2\|=\|e_1+e_2\|=\|e_1-e_2\|=1.$ It is not hard to see that $P(\tau )$ does not admit such a pair.
\end{eg}

\section*{Acknowledgments}
Bhat gratefully acknowledges funding from  ANRF  (India) through J C
Bose Fellowship No. JBR/2021/000024.
Astrid is thankful to Indian Statistical Institute Bangalore for the kind hospitality during her visits.


\begin{thebibliography}{99}

\vspace{0.5cm}


\bibitem{Ar} W. Arveson, {\em Noncommutative dynamics and $E$-semigroups.\/} Springer, (2003).

\bibitem{AL} R. Alicki, K. Lendi, {\em  Quantum Dynamical Semigroups and Applications},  Berlin, Springer Verlag (1987).

\bibitem{Bha96}  B. V. R. Bhat, An index theory for quantum dynamical semigroups, Trans. Amer. Math. Soc. {\bf 348} (1996), no. 2,
561-583. %MR1805844.
\bibitem{Bha99}  B. V. R. Bhat, Minimal dilations of quantum dynamical semigroups to semigroups of
endomorphisms of $C^*$-algebras, J. Ramanujan Math. Soc. {\bf 14}
(1999), no. 2, 109-124.
\bibitem{Bh3} B. V. R. Bhat,  Atomic dilations, In G.L. Price, B .M. Baker, P.E.T. Jorgensen, and P.S. Muhly, editors, Advances in quantum dynamics, {\bf 335}  in Contemp. Math.,  Amer. Math. Soc. (2003) 99-107.



\bibitem{BT} B. V. R. Bhat and T. Bhattacharyya, {\em 
Dilations, completely positive maps and geometry.\/}
Texts Read. Math., 84, Hindustan Book Agency, New Delhi; Springer, Singapore, (2023)



\bibitem{BBDR} B. V. R. Bhat, P. Bikram, S. De and N. Rakshit, Poisson boundary on full Fock
space, Trans. Amer. Math. Soc. {\bf 375} (2022), 5645-5668.


\bibitem{BKT1} B. V. R. Bhat, S. Kar and B. Talwar, Peripheral
Poisson boundary, To appear in  Israel J. of Math. (Arxiv:
2209.07731)

\bibitem{BKT2} B. V. R. Bhat, S. Kar and B. Talwar, Peripherally
automorphic unital completely positive maps, Linear Alg. and its
Appli. {\bf  678} (2023)191-205.


\bibitem{BP95} B. V. R. Bhat\ and\ K. R. Parthasarathy, Markov dilations of nonconservative dynamical semigroups and a quantum boundary theory, Ann. Inst. H. Poincar\'{e} Probab. Statist. {\bf 31} (1995), no.~4, 601--651. 

\bibitem{BS}  B. V. R. Bhat and M. Skeide,  Tensor product systems of Hilbert modules and dilations of completely positive semigroups, Infin. Dimens. Anal. Quantum Probab. Relat. Top.  {\bf 3} (2000), no.4,
519-575. 


 \bibitem{RA} R. Carbone, and A. Jencova,  On period, cycles and fixed points of a quantum channel,  Ann. Henri Poincare (2020), 155-188.


\bibitem{CE} M. D. Choi\ and\ E. G. Effros, Injectivity and operator spaces, J. Funct. Anal. {\bf 24} (1977), no.~2, 156--209.  

\bibitem{Da} E. B. Davies, {\em  Quantum Theory of Open Systems},  London, Academic Press (1976).

\bibitem{DP} S. Das and J. Peterson, Poisson boundaries of $II_1$
factors, Compositio Mathematica, {\bf 158}, (2022),no.~8,1746-1776. 

\bibitem{Fa} F. Fagnola,  Quantum Markov semigroups and quantum flows, Proyecciones.{\bf 18} (1999) no.3,  1–144.

\bibitem{FOR} F. Fidaleo, F. Ottomano, and S. Rossi, Spectral and ergodic properties of completely positive maps and decoherence, Linear Algebra Appl. {\bf 633} (2022), 104-126. 

\bibitem{G} U. Groh, Spectrum and asymptotic behaviour of completely positive maps on $\mathscr{B}(\mathcal{H}),$ Math. Japanica, {\bf 29},  (1984), No. 3,  395-402. 

\bibitem{GS} D. Goswami and K. B. Sinha, {\em Quantum Stochastic Processes and Noncommutative Geometry\/}, Cambridge Univ. Press (2007).

\bibitem{Izumi2002} M. Izumi, Non-commutative Poisson boundaries and compact quantum group actions, Adv. Math. {\bf 169} (2002), no.~1, 1--57. MR1916370

\bibitem{Izumi2004} M. Izumi, Non-commutative Poisson boundaries, Discrete Geometric Analysis, 69-81, Contemp. Math. 347 Amer. Math. Soc. Providence, RI 2004.

\bibitem{MS} P. S. Muhly\ and\ B. Solel, Quantum Markov processes (correspondences and dilations), Internat. J. Math. {\bf 13} (2002), no.~8, 863--906. MR1928802



\bibitem{Powers} R. T. Powers, Construction of  $E_0$-semigroups of  $\mathscr{B}(\mathcal{H})$  from CP-flows, Contemp. Math., {\bf 335}
Amer. Math. Soc., Providence, RI, 2003, 57–97.

\bibitem{Sawada} Y. Sawada, A Connes correspondence approach to the dilation theory: Internat. J. Math. 31 (2020), no. 5, 2050040, 30 pp.
\bibitem{SS} O. Shalit, M. Skeide,  CP-Semigroups and Dilations, Subproduct Systems and Superproduct Systems: the Multi-Parameter Case and Beyond, Dissertationes Mathematicae, Vol. 585 (2023), 1-233.




\end{thebibliography}
\end{document}